\documentclass{amsart}
\usepackage{amsmath,amsthm, amscd}
\usepackage{amsfonts}
\usepackage{txfonts}
\usepackage{marvosym}
\usepackage{graphicx}
\usepackage{color}
\usepackage{times}
\usepackage{epsfig}
\usepackage{fullpage}
\usepackage{setspace}
\theoremstyle{plain}

\def\@opargbegintheorem#1#2#3{\par\addvspace{6pt plus3pt minus2pt}%
    \def\@tempa{#3}%
    \noindent{\bf #1 #2 \ifx\@tempa\empty\unskip\else\unskip\ (#3).\fi\hskip.5em}\csname#1font\endcsname\ignorespaces
\ignorespaces}
\def\@endtheorem{\par\addvspace{6pt plus3pt minus2pt}}
\def\@begintheorem#1#2#3{\par\addvspace{8pt plus3pt minus2pt}%
              \noindent{\csname#1headfont\endcsname#1\ \ignorespaces#3 #2.}%
              \csname#1font\endcsname\hskip.5em\ignorespaces}
\def\@endtheorem{\par\addvspace{8pt plus3pt minus2pt}\@endparenv}
\newtheorem{theorem}{Theorem}[section]   %DON'T HAVE TO TAKE OUT [section]

\newtheorem{lemma}[theorem]{Lemma}
\newtheorem{proposition}[theorem]{Proposition}

\newtheorem{remark}[theorem]{Remark}

\DeclareMathAlphabet      {\mathbfit}{OML}{cmm}{b}{it}
\newcommand{\R}{\ensuremath{\mathbb{R}}}
\newcommand{\G}{\ensuremath{\mathcal{G}}}
\begin{document}

\title{ON CHIRALITY OF TOROIDAL EMBEDDINGS OF POLYHEDRAL GRAPHS}

\author{SENJA BARTHEL}

\address{EPFL Valais\\
Rue de l'Industrie 17\\1951 Sion, Switzerland}
              \email{senja.barthel@epfl.ch}

\maketitle

\begin{abstract}
We investigate properties of spatial graphs on the standard torus. It is known that nontrivial embeddings of planar graphs in the torus contain a nontrivial knot or a nonsplit link due to~\cite{short},\cite{explicit}. Building on this and using the chirality of torus knots and links~\cite{Murasugi1},\cite{Murasugi2}, we prove that nontrivial embeddings of simple 3-connected planar graphs in the standard torus are chiral. For the case that the spatial graph contains a nontrivial knot, the statement was shown by Castle~\textit{et al}~\cite{Hyde}.  We give an alternative proof using minors instead of the Euler characteristic. To prove the case in which the graph embedding contains a nonsplit link, we show the chirality of Hopf ladders with at least three rungs, thus generalising a theorem of Simon~\cite{Simon}.
\end{abstract}

\keywords{topological graphs; knots and links; chirality; topology and chemistry; templating on a toroidal substrate}

%\ccode{Mathematics Subject Classification 2000: 57M25, 05C10, 92E10}

\section{Introduction}
\label{intro}

The collaboration between mathematicians working in knot theory and topological graph theory and chemists working in stereochemistry has been very fruitful so far (\cite{SauvageAmabilino}-\cite{Flapan}). The spatial arrangement of a molecule can be modeled by a \textbf{spatial graph} $\mathcal{G}$, which is the image of an embedding $f:G\rightarrow \R^3$ of an abstract graph $G$ into $\R^3$ up to \textbf{ambient isotopies}; i.e., bending, stretching and shrinking of $\mathcal{G}$ without self-intersections is allowed as long as no edge is collapsed. The value of modeling a molecular structure by a spatial graph lies in the fact that topological properties of the spatial graph are inherited by the molecule. For example, topological chirality implies chemical chirality. A spatial graph is (topologically) \textbf{chiral} if it is not ambient isotopic to its mirror image. It is \textbf{achiral} otherwise; this is equivalent to the existence of an orientation-reversing homeomorphism of $\R^{3}$ that maps the spatial graph onto itself. A molecule with an underlying chiral graph is automatically a chiral molecule.

Castle, Evans and Hyde~\cite{Hyde} proved that polyhedral toroidal molecules which contain a nontrivial knot are chiral. A \textbf{polyhedral} molecule has an underlying graph which is planar, 3-connected and simple. A graph is \textbf{planar} if there exists an embedding of the graph in the sphere~$S^2$ (or equivalently in the plane~$\R^2$). Such an embedding is a \textbf{trivial embedding} and its image is a \textbf{trivial spatial graph}. A graph is \textbf{\textit{n}-connected} if at least $n$ vertices and their incident edges have to be removed to decompose the graph or to reduce it to a single vertex. A graph is \textbf{simple} if it has neither multiple edges between a pair of vertices nor loops from a vertex to itself. Nontrivial spatial graphs (as well as their corresponding molecules) which embed in the standard torus are called \textbf{toroidal}.

The argument given in~\cite{Hyde} to show the chirality of polyhedral toroidal molecules which contain a nonsplit link depends partly on a theorem of Simon~\cite{Simon} whose conditions unfortunately are not satisfied in~\cite{Hyde}. Simon's theorem states that the Hopf ladder with at least three rungs is chiral, assuming that sides are taken to sides. The \textbf{Hopf ladder $H_{n}$} is the spatial graph which is obtained from a Hopf link by adding $n$ pairs of vertices $(v^{1}_{1}, v^{2}_{1}), \dots , (v^{1}_{n}, v^{2}_{n})$, where the vertices $v^{1}_{1}, \dots , v^{1}_{n}$ lie on one link component and the vertices $v^{2}_{1}, \dots , v^{2}_{n}$ lie on the other, and by adding $n$ edges called \textbf{rungs} $e_{1}, \dots , e_{n}$, where $e_{i}$ has endpoints $v^{1}_{i}, v^{2}_{i}$ so that the rungs do not introduce any crossings in a diagram of the spatial graph as illustrated for $H_{2}$ and $H_{3}$ in Fig.~\ref{Hopfladders}. We generalise Simon's theorem with the following proposition:
\begin{proposition}\label{prop}
The Hopf ladder~$H_{n}$ with $n$ rungs is $\begin{cases} achiral &\mbox{if }\; 0 \leq n \leq 2, \\ 
 chiral & \mbox{if }\; 3 \leq n. \end{cases}$
\end{proposition}

This allows us to complete the proof in~\cite{Hyde} using a different method (minors instead of the Euler characteristic) and to obtain the following result:
\begin{theorem}[Chirality] \label{chiral}{\hspace{1mm}} \\
Nontrivial simple 3-connected planar graphs which are embedded in the standard torus~$T^2$ are chiral.
\end{theorem}

For the proof we rely on the fact that all planar toroidal spatial graphs contain a nontrivial knot or a nonsplit link:

\begin{theorem}[Existence of knots and links~\cite{short},\cite{explicit}] \label{goal}{\hspace{1mm}} \\
Let $G$ be an planar graph and $f:G\rightarrow \R^3$ be a nontrivial embedding of $G$ with image~$\mathcal{G}$.  If $\mathcal{G}$ is contained in the torus~$T^2$, it contains a subgraph which is a nontrivial knot or a nonsplit link.
\end{theorem}

\section{The proofs}
The first part of this section contains the outline of the proof of Theorem~\ref{chiral} and preparations for it, in particular the generalisation of Simon's theorem on the chirality of Hopf ladders. The second subsection proves Theorem~\ref{chiral}.
\subsection{Outline of the proof and chirality of Hopf ladders}
The proof of Theorem~\ref{chiral} uses Theorem~\ref{goal}, Proposition~\ref{prop} and the following three statements.
\begin{theorem}[Chirality of torus knots and torus links~\cite{Murasugi1},\cite{Murasugi2}] \label{Murasugi}{\hspace{1mm}}\\
Torus knots and torus links with at least three crossings are chiral.
\end{theorem}
\begin{remark}[A necessary condition for a spatial graph to be achiral]\label{K}
Let $K(\mathcal{G})$ be the set of all knots and links up to ambient isotopy contained as subgraphs in the spatial graph~$\mathcal{G}$. If $\mathcal{G}$ is achiral, every topological chiral element in the set $K(\mathcal{G})$ can be deformed into the mirror image of an element in the set $K(\mathcal{G})$.
\end{remark}

Remark~\ref{K} says that if an achiral spatial graph $\mathcal{G}$ contains a chiral spatial subgraph $\mathcal{K}$, it must also contain the mirror image $\mathcal{K'}$ of $\mathcal{K}$. Note that $\mathcal{K}$ and $\mathcal{K'}$ need not be disjoint in general but are allowed to share edges and points.

To set the stage for Theorem~\ref{Kuratowski} and the proofs in Section~\ref{proofs}, we introduce minors:
An abstract graph~$G'$ is a \textbf{minor} of $G$ if it is obtained from $G$ by a sequence of deleting and contracting edges and deleting isolated vertices.
Similarly, a spatial graph $\G'$ is a minor of a spatial graph $\G$ if $\G'$ is obtained from $\G$ by deletion and contraction of edges and deletion of isolated vertices. \textbf{Contraction} along an edge $e$ of a spatial graph $\G$ means shrinking $e$ to a point while keeping the edges which are attached to the endpoints of $e$ attached. Contraction of an edge is only defined for edges which are not loops.
\begin{theorem}[Planarity criterion~\cite{Kuratowski}]\label{Kuratowski}{\hspace{1mm}}\\
A graph is planar if and only if it contains neither $K_5$ nor $K_{3,3}$ as minors.
\end{theorem}
\vskip 5pt
\noindent
\textbf{Outline of the proof of Theorem~\ref{chiral}:}\\
The idea of the proof is to see that a simple $3$-connected planar~spatial graph $\G$ which is nontrivially embedded in $T^2$ contains a chiral subgraph which cannot be extended to an achiral spatial graph by adding vertices and edges on the torus without losing its planarity. Theorem~\ref{goal} ensures the existence of nontrivially knotted or nonsplit linked subgraphs of $\G$. This allows a proof in two cases: Case~1 deals with the case in which $\G$ contains a nontrivial knot whereas Case~2 covers the case in which $\G$ contains a nonsplit link. The second case has to further distinguish between the possibilities that either the nonsplit link is different from the Hopf link (Case~2a) or the Hopf link (Case~2b).

In Case~1, the knot is chiral by Theorem~\ref{Murasugi}. For $\G$ to be achiral, the mirror image of the knot must also be a subgraph of $\G$ by Remark~\ref{K}. Therefore, we extend any nontrivial torus knot to a spatial graph~$\mathcal{K}$ so that the criterion of Remark~\ref{K} is satisfied. We show in Lemma~\ref{L1} that $K_5$ is a minor of every $\mathcal{K}$ which is constructed in such a way. It follows from Theorem~\ref{Kuratowski} that~$\mathcal{K}$ is nonplanar. This proves Theorem~\ref{chiral} if $\G$ contains a nontrivial knot. 

In Case~2a, the same argument as for the knotted case proves the theorem. The only difference is that here $K_{3,3}$ (instead of $K_5$) is a minor of the extension $\mathcal{K}$ as shown in Lemma~\ref{L2}.

In the remaining Case~2b, simpleness and 3-connectivity ensure the existence of the Hopf ladder with three rungs (Fig.~\ref{Hopfladders}, right) as shown by~\cite{Hyde}. The Hopf ladder with three rungs is chiral as shown in Proposition~\ref{prop}. Here Lemma~\ref{L3} ensures that $K_{3,3}$ is a minor of any achiral extension $\mathcal{K}$ of the Hopf ladder with three rungs.
\vskip 5pt
The idea of the argument is similar to the one given by Castle, Evans and Hyde~\cite{Hyde} who showed the statement of Theorem~\ref{chiral} in the case that the spatial graph contains a nontrivial knot. But while the argument in~\cite{Hyde} uses the Euler characteristic, we detect one of Kuratowski's nonplanar minors $K_{3,3}$ and $K_5$ in the achiral extensions. There is a gap in the proof in~\cite{Hyde} for the case that the spatial graph contains a nonsplit link: The argument given there depends on the mistaken presumption that Simon~\cite{Simon} proved the chirality of Hopf ladders with at least three rungs. Unfortunately, Simon's proof assumes that rungs go to rungs and sides go to sides, which is not given in general. We fill the gap by showing the chirality of the Hopf ladder with at least three rungs in Proposition~\ref{prop} without making Simon's extra assumptions. 
\vskip 10pt
\noindent
\textbf{Proposition~\ref{prop}}\\
\textit{The Hopf ladder~$H_{n}$ with $n$ rungs is $\begin{cases} achiral &\mbox{if }\; 0 \leq n \leq 2, \\ 
 chiral & \mbox{if }\; 3 \leq n. \end{cases}$}
\begin{proof}
The Hopf link~$H$ is achiral as unoriented link but chiral if oriented, which can be easily confirmed by calculating the linking number. One or two rungs do not determine an orientation on the link. Since there exists a symmetric representation of the Hopf ladder with one or two rungs (Fig.~\ref{Hopfladders}, left), these Hopf ladders are achiral. We prove the chirality of the Hopf ladder $H_{n}$ with $n \geq 3$ rungs by contradiction. Recall that a spatial graph is achiral if there exists an orientation-reversing homeomorphism~$h$ of~$\R^3$ which maps the spatial graph onto itself. Assume there exists an orientation-reversing homeomorphism~$h$ of~$\R^3$ with $h(H_{n})=H_{n}, n \geq 3$. Every homeomorphism~$h$ maps nontrivially linked subgraphs onto nontrivially linked subgraphs. Since the Hopf link~$H$ which consists of the sides of~$H_{n}$ is the only nontrivially linked subgraph of~$H_{n}$, it follows that~$H$ is mapped onto itself, i.e., $h(H)=H$. Since $h$ reverses the orientation, we have $h(H)=H^{\star}$, where $H^{\star}$ is the mirror image of $H$. The Hopf ladder $H$ is chiral since more than two rungs $e_{1}, \dots , e_{n}$  with endpoints $v_{1}, v'_{1}, \dots , v_{n}, v'_{n}$ determine orientations on the sides of the Hopf ladder (Fig.~\ref{Hopfladders}, right: $v_{1}, \dots , v_{n}= 1,2,3$, $v'_{1}, \dots , v'_{n}= 4,5,6$) by the order of the endpoints of the rungs $v_{1}, \dots , v_{n}$ respectively $v'_{1}, \dots , v'_{n}$, therefore implying that $H$ is different from~$H^{\star}$. This contradicts the existence of $h$ with $H=h(H)=H^{\star}$ and it follows that the Hopf ladder $H_{n}$, with $n \geq 3$, rungs is chiral.
\end{proof}

\begin{figure}[th]
	\centering
	\def\svgwidth{350pt}
	 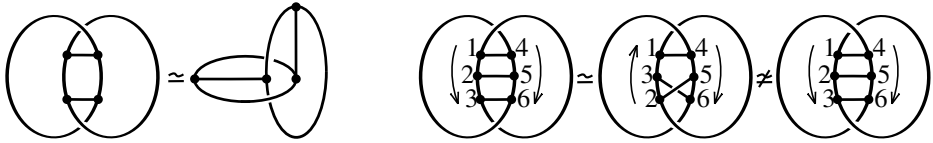
\vspace*{8pt}
\caption{Left: The Hopf ladder with less than three rungs is achiral. It is neither simple nor 3-connected. Right: The Hopf ladder with three or more rungs is chiral by~Proposition~\ref{prop}.}
\label{Hopfladders}
\end{figure}

\subsection{Proof of Theorem~\ref{chiral}}\label{proofs}
The following three lemmas will be used in the proof of Theorem~\ref{chiral} to show that achiral extensions of knots, chiral links and $H_{3}$ on the torus fail to be planar. 
\begin{lemma}{\hspace{1mm}}\label{L1}\\ 
\textit{$K_5$ is a minor of every spatial graph on the torus which contains both knots $T(p,q)$ and $T(r,-s)$ for some relatively prime integers $p\geq 2,q\geq 3$ and relatively prime integers $r\geq 1,s\geq 1$.}
\end{lemma}
\begin{proof}
Let $T(a,b)\otimes T(c,d)$ be the spatial graph which is constructed by embedding $T(a,b)$ and $T(c,d)$ with minimal number of intersections in the torus and by adding vertices at the intersection points. (For example, $T(2,3)\otimes T(2,-5)$ is drawn on the left of Fig.~\ref{ErsterFall} and $T(2,3)\otimes T(1,-1)$ is drawn on the right.) 
View the torus as the rectangle $[0,1] \times [0,1]$ with opposite sides identified. A meridian is given as $\{x\} \times [0,1]$ and a longitude as $[0,1] \times \{y\}$. Without loss of generality, arrange $T(p,q)$ and $T(r,-s)$ on the rectangle in such a way that no vertex lies on the boundary of the rectangle and that $T(r,-s)$ runs through its corner point.

Then there exists a path $\pi$ in $T(p,q) \otimes T(r,-s)$ from the corner point $(0,1)$ to the corner point $(1,0)$ of the rectangle which does not intersect the boundary of the rectangle in any other points and which respects the orientations of $T(p,q)$ and $T(r,-s)$ (fat zig-zag in Fig.~\ref{ErsterFall}, left). The existence of $\pi$ can be seen as follows (compare Fig.~\ref{ErsterFall} for notation): If $r=s=1$, set $\pi = T(r,-s)$. Otherwise, there are exactly two segments of $T(r,-s)$ in the interior of the rectangle which are connected to the corner point. Denote by $S_{1}$ the segment with endpoints $(0,1)$ and $s_{1}$ and denote by $S_{n}$ the segment with endpoints $s_{n}$ and $(1,0)$. By construction, $s_{1}$ lies on $]0,1[ \times \{0\}$ and $s_{n}$ on $]0,1[ \times \{1\}$. Consider the region $B$ of the rectangle which is bounded by the two segments $S_{1}$ and $S_{n}$, $[(0,1), s_{n}]$ and $[s_{1}, (1,0)]$ (shaded region in Fig.~\ref{ErsterFall}). It follows that $T(r,-s) \cap \mathring{B}$ is either empty or the union of segments $S_{i}$, $1 < i <n$ running parallel to $S_{1}$ and $S_{n}$, each having one endpoint in $](0,1), s_{n}[$ and the other in $]s_{1}, (1,0)[$. Therefore, $B$ is a union of bands $b_{i}$ which are bounded by consecutive segments $S_{i}$ and $S_{i+1}$, $1 \leq i \leq n-1$, and intervals in $[(0,1), s_{n}]$ and $[s_{1}, (1,0)]$ between the endpoints of $S_{i}$ and $S_{i+1}$. Since $T(p,q)$ is a nontrivial torus knot, $T(p,q) \cap b_{i} \neq \emptyset$ for all $1 \leq i \leq n-1$. Since $T(p,q)$ follows both the longitude and the meridian of the torus with positive orientation and since $T(r,-s)$ follows the longitude with positive and the meridian with negative orientation, $T(p,q) \cap b_{i}$ does not run parallel to $S_{i}$ or $S_{i+1}$, and we conclude that there exists a connected component of $T(p,q) \cap b_{i}$ which connects $S_{i}$ and $S_{i+1}$. This is true for all $b_{i}, 1 \leq i \leq n-1$, hence a path from $S_{1}$ to $S_{n}$ along $T(p,q) \otimes T(r,-s)$ inside $]0,1[ \times ]0,1[$ is found. Note that this path does not intersect the boundary of the rectangle since no vertex of $T(p,q) \otimes T(r,-s)$ lies on the boundary of the rectangle by assumption. Extending the path along $S_{1}$ towards $(0,1)$ and along $S_{n}$ towards $(1,0)$, gives the desired path $\pi$. 
(Fig.~\ref{ErsterFall}, left).
\begin{figure}[th]
	\centering
\def\svgwidth{350pt}
	 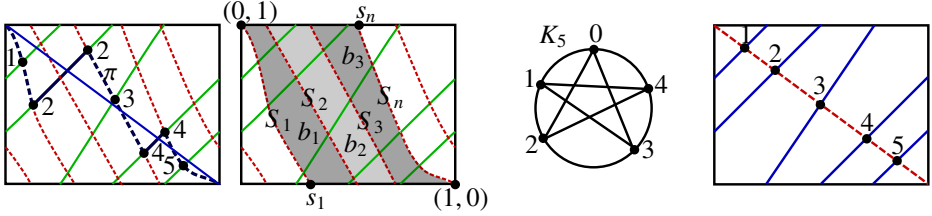
\vspace*{8pt}
\caption{Illustration for Lemma~\ref{L1}. Left: $T(p,q) \otimes T(r,-s)$ with $p=2, q=3, r=2, s=5$. Here $T(p,q)$ is drawn in light green (full), $T(r,-s)$ in red (dashed), $T(1,-1)$ in blue (full diagonal). A possible path $\pi$ is drawn in dark blue (very fat). Second: Showing the existence of the path $\pi$. $B$ is the  shaded region. Third: The graph $K_5=S(2,3)$.
Right: $K_5$ represented as $T(2,3) \otimes T(1,-1)$ is a minor of a spatial graph in the torus which contains both knots $T(p,q)$ and $T(r,-s)$. Blue (full): $T(2,3)$, red (dashed): $T(1,-1)$.}
\label{ErsterFall}
\end{figure}

Now construct a minor of $T(p,q) \otimes T(r,-s)$ as follows: Delete the edges of $T(r,-s)$ which are not part of $\pi$, and contract the edges of $T(p,q) \cap \pi$. This gives the spatial graph $T(p,q) \otimes T(1,-1)$, which is an embedding of the abstract graph~$S(p,q)$ (compare Fig.~\ref{ErsterFall2}). $S(p,q)$ can be described as a cycle $C$ on which $p+q$ vertices $\{v_{0}, \dots v_{p+q-1} \}$ are placed, together with additional $p+q$ edges $c_{j}$ which are chords connecting the vertices $v_{j}$ and $v_{j+p}\text{mod}(p+q)$, $0 \leq j \leq (p+q-1)$. Note that $S(p,q)=S(q,p)$.

It is left to show that $K_5$ is a minor of $S(p,q)$, which itself is a minor of $T(p,q) \otimes T(r,-s)$ (Fig.~\ref{ErsterFall2}).
Note that $T(q,p) \otimes T(1,-1)$ leads to the same graph $S(p,q)$, allowing us to restrict the argument to graphs $T(p,q) \otimes T(1,-1)$ with $p<q$. Also, recall that $p,q$ are relatively prime integers and $p \geq 2$. The argument works inductively. We present the argument in its entirety and add the details of the two single steps `taking a minor' and `inversion plus renaming' at the end of the proof.

Set $p_{0}=p,q_{0}=q$. If $p_0=2, q_0=3$, we are done since $S(2,3)=K_5$. For $p_{i}>2$, we construct a proper minor (i.e., a minor which is not the graph itself) of $S(p_{i},q_{i})$ which is of the form $S(p_{i,M}, q_{i,M})$ with $p_{i,M}=2 , q_{i,M}=2 \left \lfloor \frac{p_{i}+q_{i}}{p_{i}} \right\rfloor -1$.
By interchanging the roles of the cycle $C_{i,M}$ from $S(p_{i,M}, q_{i,M})$ and its chords, another representation $S(p'_{i,M}, q'_{i,M})$ of $S(p_{i,M}, q_{i,M})$ is obtained. This is possible because the chords are not edges of the cycle $C_{i,M}$ but form a closed cycle themselves by construction: $S(p,q)$ is the Cayley graph of the cyclic group $(\mathbb{Z}/(p+q),+)$ with $p+q$ elements and generator set $\{1, p\}$. Since $p$ is a generator of $\mathbb{Z}/(p+q)$ and $p,q$ are relatively prime integers, there exists a unique integer $0 \leq k<p+q$ for each $0 \leq j \leq p+q-1$ such that $k p \,\text{mod}(p+q)=j$. Relabel the vertices of the graph by $v_{j} \mapsto v_{k}$.  In particular, there exists $p'$ with $p' p \; \text{mod}(p+q)=1$, and if we set $(q+p)-p'=q'$, the relabeled graph is of the form $S(p', p'+((q+p)-p'))=S(p', p'+q')$.
$S(p'_{i,M}, p'_{i,M})$ and $S(p_{i,M}, q_{i,M})$ satisfy the relations $p_{i,M}p'_{i,M}\text{mod}(p_{i,M}+q_{i,M}) = 1$ and $q_{i,M}q'_{i,M}\text{mod}(p_{i,M}+q_{i,M}) = 1$. We show that we have either $p'_{i,M} \neq 2 \neq q'_{i,M}$ or we have $p_{i,M} = 2$ and $q_{i,M} = 3$. In the second case we are done, whereas in the first case we have completed the generic induction step: We now consider $S(p_{i+1},q_{i+1})$, which is obtained from $S(q'_{i,M}, p'_{i,M}) = S(p'_{i,M}, q'_{i,M})$ by setting $p_{i+1} := q'_{i,M} , q_{i+1} :=p'_{i,M}$. At each step, we reduce the numbers of vertices and edges in the graph, namely $p_{i}+q_{i}$ and $2(p_{i}+q_{i})$, respectively, when a minor is taken, i.e., $p_{i-1,M}+q_{i-1,M}=p_{i}+q_{i}>p_{i,M}+q_{i,M}=p_{i+1}+q_{i+1}$. In particular, since $p_{i,M}=2$, it follows that $q_{i,M}>q_{i+1,M}$ and that  for all $i$ the $q_{i,M}$ are odd numbers. We stop as soon as we either reach $p_{i}=2, q_{i}=3$ and the procedure does not give a proper minor, i.e., $S(p_{i}, q_{i})=S(p_{i,M}, q_{i,M})$, or we have $p_{i,M}=2$ and $q_{i,M}=3$, in which case $S(p_{i,M}, q_{i,M})=S(p_{i+1}, q_{i+1})$. In any case, we reach $S(2,3)$ after finitely many steps: Since $S(2,3)=K_5$, we are done. The steps of the iteration can be pictured as follows: $$S(p_{i}, q_{i})\,\, \underrightarrow{\text{minor}}\,\, S(p_{i,M}, q_{i,M})\, \,\underrightarrow{\text{invert}}\, \,S(p'_{i,M}, q'_{i,M}) = S(q'_{i,M}, p'_{i,M})\, \,\underrightarrow{\text{rename}}\,\, S(p_{i+1}, q_{i+1}).$$

`Taking a minor': obtaining $S(p_{i,M}, q_{i,M})$ from $S(p_{i},q_{i})$ (Fig.~\ref{ErsterFall2}):
Take the subgraph of $S(p_{i},q_{i})$ which consists of $C_{i}$ and the first $2 \left\lfloor \frac{p_{i}+q_{i}}{p_{i}} \right\rfloor +1$ chords. This is precisely the number of chords needed such that each chord is intersected by exactly two other chords. We obtain the formula as follows: If $p_{i}=2$, we must do exactly two full turns along chords. Only for $p_{i}=2$, the endpoints of the chords from the second round lie exactly in the middle of the points to which the chords from the first round are attached. If the endpoints of the chords from the second round are closer to the endpoints than to the starting points of the chords from the first round, we have to stop traveling along chords just after completing the second round in order to ensure that every chosen chord is intersected by exactly two other chords. If the endpoints of the chords from the second round are closer to the starting points than to the endpoints of the chords from the first round, we have to stop traveling along chords just before finishing the second round. We can in fact determine in which case we are: $\left\lfloor \frac{p_{i}+q_{i}}{p_{i}} \right\rfloor$ is the number of chords before finishing the first round. Therefore, $\left\lfloor \frac{p_{i}+q_{i}}{p_{i}} \right\rfloor p_{i}$ is the index of the endpoint of the last chord from the first round. If this vertex is closer to $v_{0}$ than to $v_{q_{i}+1}$, we are in the first case, i.e., $(p_{i}+q_{i}) -\left \lfloor \frac{p_{i}+q_{i}}{p_{i}} \right \rfloor p_{i} < \left \lfloor \frac{p_{i}}{2}\right \rfloor$. Similarly, we are in the second case if $(p_{i}+q_{i}) - \left\lfloor \frac{p_{i}+q_{i}}{p_{i}}\right \rfloor p_{i} >\left \lfloor \frac{p_{i}}{2}\right \rfloor$. If $(p_{i}+q_{i}) -\left \lfloor \frac{p_{i}+q_{i}}{p_{i}} \right\rfloor p_{i} = \left\lfloor \frac{p_{i}}{2} \right\rfloor$, we have $p_{i}=2$. Hence, in both cases, we have traveled through the first $2 \left \lfloor \frac{p_{i}+q_{i}}{p_{i}} \right \rfloor +1$ chords $c_{0}, \dots c_{2\left \lfloor \frac{p_{i}+q_{i}}{p_{i}}\right \rfloor}$. \\
To obtain $S(p_{i,M}, q_{i,M})$ from $S(p_{i},q_{i})$ first delete the chords $c_{2\left \lfloor \frac{p_{i}+q_{i}}{p_{i}}\right \rfloor+1}, \dots , c_{p_{i}+q_{i}-1}$. Then contract iteratively all edges of the cycle $C_{i}$ which have an endpoint that is not an endpoint of a remaining chord. Contract the edge in $C_{i}$ which connects $v_{0}$ and the endpoint of $c_{2\left \lfloor \frac{p_{i}+q_{i}}{p_{i}} \right\rfloor}$. The resulting graph has $2 \left\lfloor \frac{p_{i}+q_{i}}{p_{i}}\right \rfloor +1$ vertices and $2(2\left \lfloor \frac{p_{i}+q_{i}}{p_{i}} \right\rfloor +1)$ edges by construction. Again, $2\left \lfloor \frac{p_{i}+q_{i}}{p_{i}} \right\rfloor +1$ edges form a new cycle $C_{i,M}$, whereas the other edges are chords $c_{j}$ connecting the vertices $v_{j}$ and $v_{j+2}\text{mod}(2+q_{i,M})$ for $0 \leq j \leq (2+q_{i,M}-1)$, where $q_{i,M}=2 \left\lfloor \frac{p_{i}+q_{i}}{p_{i}}\right \rfloor -1$. Therefore, the graph is of the form $S(p_{i,M}, q_{i,M})$ and $p_{i,M}=2$.

`Inversion plus renaming': obtaining $S(p_{i+1}, q_{i+1})$ from $S(p_{i,M}, q_{i,M})$ (Fig.~\ref{ErsterFall2}):
Since the chords are not edges of $C_{i,M}$ but form a closed cycle, we can change the roles of the chords and the cycle while leaving the unlabeled graph unchanged. This leads to a representation of the graph in terms of $p'_{i,M}$ and $q'_{i,M}$, which are given by $p_{i,M}p'_{i,M}\text{mod}(p_{i,M}+q_{i,M}) = 1$ and $q_{i,M}q'_{i,M}\text{mod}(q_{i,M}+q_{i,M}) = 1$, as explained above. Note that $p'_{i,M} > q'_{i,M}$ since $p_{i,M}< q_{i,M}$ and $p_{i,M}+q_{i,M}=p'_{i,M}+q'_{i,M}$. To continue the induction we need to show that $p'_{i,M} \neq 2$ as well as $q'_{i,M} \neq 2$ unless $p_{i,M} = 2$ and $q_{i,M} = 3$, in which case we have already obtained the final graph $S(2,3)$. Thus, if $p'_{i,M}=2$, we have $p_{i,M}p'_{i,M}\text{mod}(p_{i,M}+q_{i,M}) = 2 \cdot 2\text{mod}(2+q_{i,M}) =1$, implying $q_{i,M} =1$, which contradicts the assumption that $q_{i,M} >1$. On the other hand, if $q'_{i,M}=2$, we have $q_{i,M}q'_{i,M}\text{mod}(p_{i,M}+q_{i,M}) = 1$, which is equivalent to $q_{i,M}q'_{i,M}= n(p_{i,M}+q_{i,M}) +1$ for an integer $n$. Since $q'_{i,M}=2$ by assumption, this implies $2q_{i,M}= n(p_{i,M}+q_{i,M})+1$, and since $p_{i,M}=2$ and $q_{i,M}$ is a positive odd integer, this is equivalent to $q_{i,M} = \frac{2n+1}{2n-1}$with $n=1$. Consequently, $q_{i,M}=3$ and therefore $S(p_{i,M}, q_{i,M})= S(2,3)$. 

Hence, $K_5$ is a minor of $S(p,q)$, which in turn is a minor of $T(p,q) \otimes T(r,-s)$. The obtained embedding of $K_5$ in the torus is $T(2,3) \otimes T(1,-1)$. This completes the proof.
\end{proof}
\begin{figure}[th]
	\centering
\def\svgwidth{420pt}
	 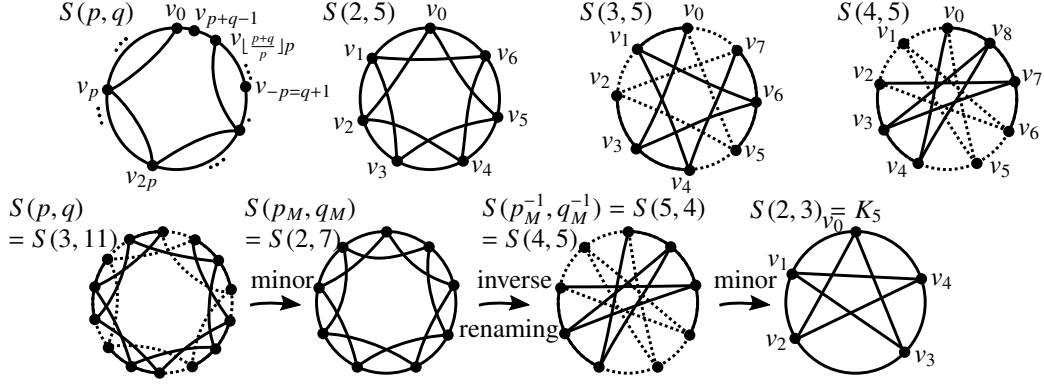
\vspace*{8pt}
\caption{Illustration for Lemma~\ref{L1}. Top: the graph $S(p,q)$ and examples for the cases $(p+q) - \lfloor \frac{p+q}{p} \rfloor p = \lfloor \frac{p}{2} \rfloor$, $(p+q) - \lfloor \frac{p+q}{p} \rfloor p > \lfloor \frac{p}{2} \rfloor$, $(p+q) - \lfloor \frac{p+q}{p} \rfloor p < \lfloor \frac{p}{2} \rfloor$.
Bottom: The procedure of constructing the minor $K_{5}$ from $S(3,11)$. The dashed edges on the circle are contracted and dashed chords are deleted.
 }
\label{ErsterFall2}
\end{figure}
\begin{lemma}{\hspace{1mm}}\label{L2}\\
\textit{$K_{3,3}$ is a minor of every spatial graph on the torus which contains both links $T(kp,kq)$ and $T(kp,-kq)$ for some relatively prime integers $p$ and $q$ and some $k \geq 2$.}
\end{lemma}
\begin{proof}
Let $p$ and $q$ be relatively prime integers and let $k \geq 2$. Then $T(kp,-kq)$ contains two different connected components of the form $T(p,-q)$, which we denote by $c_1$ and $c_2$, respectively. The components $c_1$ and $c_2$ run parallel on the torus. Likewise, $T(kp,kq)$ contains two different connected components denoted by $c'_1$ and $c'_2$ of the form $T(p,q)$, which run parallel to each other (see Fig.~\ref{FallZwei} for notations). Since $T(p,q)$ follows the longitude and the meridian of the torus with the same orientation and $T(p,-q)$ follows the longitude and the meridian with opposite orientation, $c'_1$ intersects $c_1$ in at least two points. Take one segment $(p_{1}, p_{2})$ of $c'_1$ which intersects $c_1$ only in its endpoints denoted by $p_1$ and $p_2$. If one follows $c_1$ from $p_1$, the next intersection point denoted by $p_3$ is with $c'_2$. Let $(p_{1}, p_{3})$ be the segment of $c_1$. Then there starts a segment $(p_{3}, p_{4})$  of $c'_2$ which runs parallel to $(p_{1}, p_{2})$ and intersects $c_1$ only in its endpoints $p_3$ and $p_4$. Let $(p_{2}, p_{4})$ be the segment of $c_2$ which runs from $p_2$ to $p_4$. This determines a disc in the torus which is bounded by the edges $(p_{1}, p_{2})$, $(p_{2}, p_{4})$, $(p_{3}, p_{4})$, and $(p_{1}, p_{3})$. Since $c_2$ runs parallel to $c_1$, it cuts the boundary of the disc in exactly two points $p_5$ on $(p_{1}, p_{2})$ and $p_6$ on $(p_{3}, p_{4})$. Following $c_1$, one passes through the points $p_1$, $p_3$, $p_2$, $p_4$ in this order. This determines the subgraph with the six vertices $p_{1}, \dots , p_{6}$ and the edges $(p_{1}, p_{3})$, $(p_{3}, p_{2})$, $(p_{2}, p_{4})$, $(p_{4}, p_{1})$, $(p_{1}, p_{5})$, $(p_{5}, p_{2})$, $(p_{3}, p_{6})$, $(p_{6}, p_{4})$ and $(p_{5}, p_{6})$, which is precisely $K_{3,3}$. Being a subgraph, $K_{3,3}$ is in particular a minor of every spatial graph on the torus which contains both links $T(kp,kq)$ and $T(kp,-kq)$.  
\end{proof}

\begin{figure}[th]
	\centering
\def\svgwidth{350pt}
	 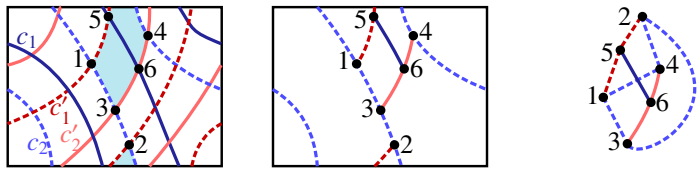
\vspace*{8pt}
\caption{Illustration of Lemma~\ref{L2}. Left: the spatial graph $T(kp,kq) \otimes T(kp,-kq)$, with \mbox{$p=1, q=2, k=2$}. Middle and right: $K_{3,3}$ is a minor  of $T(kp,kq) \otimes T(kp,-kq)$.}
\label{FallZwei}
\end{figure}

\begin{lemma}{\hspace{1mm}}\label{L3}\\
\textit{$K_{3,3}$ is a minor of every achiral spatial graph on the torus which contains the Hopf ladder~$H_3$.}
\end{lemma}
\begin{proof}
Recall that a spatial graph is achiral if and only if there exists an orientation-reversing homeomorphism of $\R^{3}$ which maps the spatial graph to itself. Such a homeomorphism induces a graph isomorphism on the graph. Graph isomorphisms map cycles to cycles, and graph isomorphisms which are induced by orientation-reversing homeomorphisms invert the sign of the linking number of links and their images. 
Consequently, every achiral graph which contains an oriented link also contains a second link which is its mirror image. The sides of the Hopf ladder~$H_3$ form an oriented Hopf link, without loss of generality realized by $T(2,2)$, where the orientation is induced by the order of the endpoints of the rungs as described in the proof of Proposition~\ref{prop}. Keeping this in mind, we find that there are precisely two options to prevent a graph which contains $H_3$ as a subgraph from inheriting the chirality of the oriented Hopf link: Firstly, it can contain a second oriented link $T(2,-2)$ which is the mirror image of $T(2,2)$. Note that $T(2,-2)$ has to be added to $H_3$ since $T(2,2)$ is its only nontrivially linked subgraph. Since such an extension contains $T(2,2) \otimes T(2,-2)$, it follows from Lemma~\ref{L2} that $K_{3,3}$ is a minor of this extension. Or, secondly, the orientation of the Hopf link can be eliminated by adding a further rung to the Hopf ladder. In that case, we find that the resulting abstract graph is nonplanar, which is independent of the choice of embedding (Fig.~\ref{Hopftrafo2}).
An extra edge with vertices $0$ and $Z$ is added whose vertex $0$ lies on an edge of one component of the Hopf link  -- without loss of generality on the edge with endpoints 1 and 2, and whose vertex $Z$ lies in the interior of one of the edges with endpoint 3 on the other component of the Hopf link -- without loss of generality on the edge with vertices $B$ and $C$. 
The Hopf ladder~$H_3$ with the extra edge $(0,Z)$ contains $K_{3,3}$ as a minor as can be seen by deleting the edge between $1$ and $3$ and contracting the edges between $0$ and $1$ and between $2$ and $3$ (Fig.~\ref{Hopftrafo2}). Although not every embedding of this extension is an achiral graph, every extension of $H_3$ which does not determine an orientation of the Hopf link contains the Hopf ladder~$H_3$ with an extra edge as constructed above. It follows that every extension of $H_3$ to an achiral graph contains a minor $K_{3,3}$.
\end{proof}

\begin{figure}[th]
	\centering
\def\svgwidth{250pt}
	 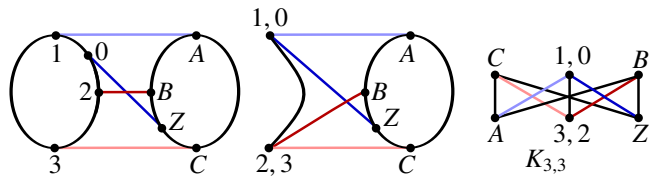
\vspace*{8pt}
\caption{Extending the abstract graph of the Hopf ladder $H_{3}$ by adding the vertices 0 and $Z$ and an edge between them as described in Lemma~\ref{L3}. The extra edge inhibits the induction of orientations on the two black circles which correspond to the Hopf link.}
\label{Hopftrafo2}
\end{figure}

\begin{remark}
It is possible to embed the Hopf ladder~$H_3$ with the extra edge $(1,Z)$ (as constructed in the proof of Lemma~\ref{L3}) as an achiral graph since there is a symmetric representation with the unoriented Hopf link mapped onto itself. The resulting spatial graph still lies on the torus but is nonplanar since $K_{3,3}$ is a minor (Fig.~\ref{Hopftrafo}). It is shown in Fig.~\ref{Hopftrafo} how one is led to adding an extra rung when constructing the symmetric representation by extending $H_3$ to an achiral graph which maps the Hopf link to its mirror image. 
\end{remark}

\begin{figure}[th]
	\centering
\def\svgwidth{300pt}
	 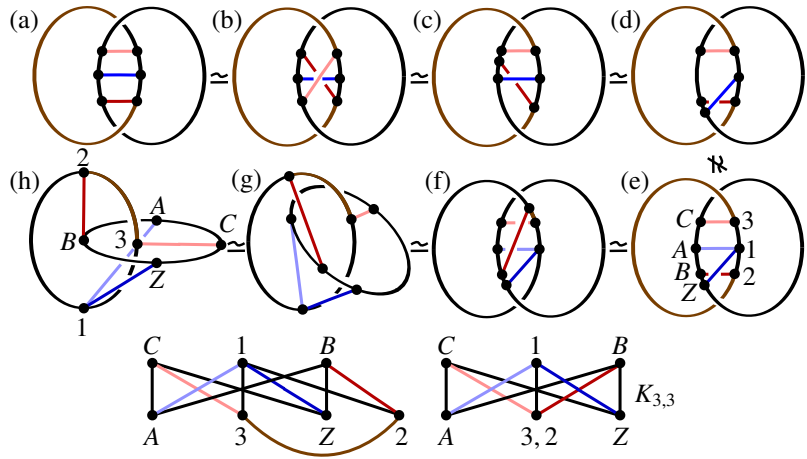
\vspace*{8pt}
\caption{Extension of the Hopf ladder $H_{3}$ to an achiral graph which contains both $H_{3}$ and its mirror image $H_{3}^{\star}$ by adding one additional edge.}
\label{Hopftrafo}
\end{figure}

We are now ready to prove Theorem~\ref{chiral}.

\begin{proof}
It follows from Theorem~\ref{goal} that the spatial graph~$\G$ contains a subgraph which is a nontrivial knot or a nonsplit link. 

\textbf{Case~1:} The spatial graph~$\G$ contains a nontrivial knot.\\
A subgraph of a spatial graph on the torus which forms a knot is a torus knot by definition. A nontrivial torus knot can be written as $T(p,q)$, for some relatively prime integers $p$ and $q$ with $|p|,|q| >1$ in order to exclude the trivial knot. A nontrivial torus knot is chiral by Theorem~\ref{Murasugi}, and its mirror image is $T(-p,q)=T(p,-q)$. If $\G$ is achiral, $\G$ contains both $T(p,q)$ and $T(p,-q)$ as subgraphs by Remark~\ref{K}.
As shown in Lemma~\ref{L1}, $K_5$ is a minor of every spatial graph which is embedded in the torus and which contains both knots $T(p,q)$ and $T(r,-s)$ for some relatively prime integers $p \geq 2,q \geq 3$ and some relatively prime integers $r\geq 1,s\geq 1$. This includes the case of a spatial graph containing both $T(p,q)$ and $T(p,-q)$ as subgraphs. The statement of Theorem~\ref{chiral} follows from Lemma~\ref{L1} since by Theorem~\ref{Kuratowski} a graph is nonplanar if and only if it contains neither $K_5$ nor $K_{3,3}$ as minors.

\textbf{Case~2:} The spatial graph~$\G$ contains a nonsplit link.\\
A subgraph of a spatial graph on the torus which forms a nonsplit link is a torus link by definition. 
A nonsplit torus link with $k$ components can be written as $T(kp,kq)$ for some relatively prime integers $p$ and $q$ with $k \geq 2$ in order to exclude the unlink.

\textbf{Case~2a:} The spatial graph~$\G$ contains a nonsplit link different from the Hopf link.\\
A nonsplit torus link which is not the Hopf link is chiral by Theorem~\ref{Murasugi}, and its mirror image is $T(-kp,kq)=T(kp,-kq)$. If $\G$ is achiral, $\G$ contains both $T(kp,kq)$ and $T(kp,-kq)$ as subgraphs by Remark~\ref{K}. As shown in Lemma~\ref{L2}, $K_{3,3}$ is a minor of every graph $\G$ which is embedded in the torus and which contains both links $T(kp,kq)$ and $T(kp,-kq)$, $k\geq 2$ as subgraphs. The statement of Theorem~\ref{chiral} now follows from Lemma~\ref{L2} and Theorem~\ref{Kuratowski}. 

\textbf{Case~2b:} The spatial graph~$\G$ contains a Hopf link.\\
In this case 3-connectivity and simpleness imply that the Hopf ladder $H_3$ with three rungs is a subgraph of $\G$.  By Proposition~\ref{prop}, $H_3$ is chiral. If $\G$ is achiral, $K_{3,3}$ is a minor of $\G$ as shown in Lemma~\ref{L3}. The statement of Theorem~\ref{chiral} now follows from Lemma~\ref{L3} and Theorem~\ref{Kuratowski}.
\end{proof}

\begin{remark}
Simpleness and 3-connectivity in Theorem~\ref{chiral} are only needed for the case in which $\G$ contains no nontrivial knot and the only nontrivially linked subgraph forms a Hopf link. If the spatial graph contains a nontrivial knot or a nonsplit link different from the Hopf link, these two assumptions can be dropped.
\end{remark}

\begin{remark}
It is not possible to weaken the assumptions of Theorem~\ref{chiral}. Counterexamples are given or can be constructed as in \cite{Hyde}.
\end{remark}

\section*{Acknowledgments}
I thank Matt Rathbun, Stephen Hyde and the anonymous reviewer for helpful comments and discussions. I thank David Rottensteiner for proof-reading and helping me organize the proof of Lemma~\ref{L1}. The early stage of this research was supported by the Roth studentship of Imperial College London mathematics department, the DAAD, and the Evangelisches Studienwerk.

\end{document}